\documentclass[11pt,a4paper,reqno]{amsart}
\usepackage[T1]{fontenc}
\usepackage[utf8]{inputenc}
\usepackage[english]{babel}
\usepackage{lmodern}
\usepackage{amsmath,amssymb,amsthm,mathtools}
\usepackage[margin=25mm]{geometry}
\usepackage{microtype}
\usepackage[unicode,colorlinks=true,allcolors=blue]{hyperref}
\hypersetup{
  pdftitle={A one-sided Auslander--Lyndon theorem},
  pdfauthor={Artur Tursunbaev},
}
\makeatletter
\def\@settitle{\begin{center}\normalfont\large\bfseries\@title\end{center}}
\makeatother
\numberwithin{equation}{section}
\newtheorem{theorem}{Theorem}[section]
\newtheorem{lemma}[theorem]{Lemma}

\newtheorem*{KSlemma}{Lemma}
\theoremstyle{remark}
\newtheorem{remark}[theorem]{Remark}
\newcommand{\Z}{\mathbb Z}
\newcommand{\Q}{\mathbb Q}
\newcommand{\core}{\operatorname{core}}
\newcommand{\rank}{\operatorname{rank}}

\title[A one-sided Auslander--Lyndon theorem]
{A one-sided Auslander--Lyndon theorem}
\author{Artur Tursunbaev}
\address{St.~Petersburg State University,
Universitetskii pr.~28,
St.~Petersburg 198504, Russia}
\email{\href{mailto:tursunbaev.art@gmail.com}{tursunbaev.art@gmail.com}}
\subjclass[2020]{Primary 20E05; Secondary 20E07, 20F14, 20J05}
\keywords{Free groups, derived subgroups, normal subgroups,
Auslander--Lyndon theorem, Schreier graphs, relation modules}
\begin{document}
\begin{abstract}
We prove that if $F$ is a noncyclic free group,
$R\triangleleft F$ and $S\le F$, then the inclusion $R'\le S'$ implies
$R\le S$. Thus the normality hypothesis on the second subgroup
can be omitted from the Auslander--Lyndon theorem. No restrictions
are imposed on ranks or indices.
\end{abstract}
\maketitle

\section{Introduction}

In 1955, Auslander and Lyndon~\cite[Theorem~3]{AL}
proved that for normal subgroups $R,S$ of a nonabelian free
group $F$, the inclusion $R'\le S'$ of derived subgroups implies $R\le S$.

The two normality hypotheses play different roles in this statement.
Dunwoody~\cite[p.~153]{Dun} constructed subgroups
$R,S\le F$ such that $S\triangleleft F$ and $R'\le S'$, but
$R\nleq S$. Thus the normality hypothesis on $R$ cannot be omitted
even when $S$ is normal.

The question of whether the normality hypothesis on $S$ can be omitted
while retaining the hypothesis $R\triangleleft F$ was explicitly
listed as open in 1999~\cite[comment on Problem~F15]{BMS1999}
and was subsequently restated in~\cite[p.~10]{KMS}.

We answer this question in the affirmative.
We use the following conventions:
\[
[a,b]=a^{-1}b^{-1}ab,
\qquad H'=[H,H].
\]

\begin{theorem}\label{thm:main}
Let $F$ be a noncyclic free group of arbitrary rank, and let
$R\triangleleft F$ and $S\le F$. Then
\[
R'\le S'\quad\Longrightarrow\quad R\le S.
\]
\end{theorem}

A closely related question concerning the equality of subgroups of a free group
with the same nontrivial derived subgroup is considered in~\cite{Tursunbaev}.

\section{Chains in Schreier graphs}

Fix a free basis $X$ of a free group $E$. For $H\le E$,
let $\Gamma_H$ denote the Schreier graph with vertices $Hg$ and
positively oriented edges
\[
Hg\xrightarrow{x}Hgx,\qquad x\in X.
\]
For an oriented edge $e$, let $o(e)$ and $t(e)$ denote its
initial and terminal vertices.

If $w\in H$, then the path in $\Gamma_H$ labeled by $w$
and starting at $H$ is closed. Denote by
\[
\sigma_H(w)\in C_1(\Gamma_H;\mathbb Z)
\]
its associated $1$-chain, with signs determined by the direction of traversal.
Free reduction of the word does not change $\sigma_H(w)$, and for
$u,v\in H$ we have
\[
\sigma_H(uv)=\sigma_H(u)+\sigma_H(v).
\]
Thus we obtain a well-defined homomorphism
\[
\sigma_H:H\longrightarrow C_1(\Gamma_H;\mathbb Z).
\]
Since
\[
H_1(\Gamma_H;\mathbb Z)=Z_1(\Gamma_H;\mathbb Z)
\subseteq C_1(\Gamma_H;\mathbb Z),
\]
the standard identification $\pi_1(\Gamma_H,H)\cong H$, followed by
abelianization, gives
\begin{equation}\label{eq:sigma-kernel}
\ker\sigma_H=H'.
\end{equation}

\begin{lemma}\label{lem:support}
Let $E$ be a free group, and suppose that $R\triangleleft E$, $S\le E$,
$E=RS$ and $R'\le S'$. Let
$c\in(R\cap S)\setminus S'$ and let
\[
w=s_1\cdots s_\ell,\qquad s_i\in X^{\pm1},
\]
be a word representing $c$ with respect to a free basis $X$
of $E$. Then
\[
[E:S]\le\ell.
\]
\end{lemma}

\begin{proof}
Consider the Schreier graph $\Gamma_S$ with respect to the basis $X$.
By~\eqref{eq:sigma-kernel},
\[
\alpha=\sigma_S(c)\ne0.
\]
For every $r\in R$, we have
\[
rcr^{-1}c^{-1}\in R'\le S'.
\]
Hence $rcr^{-1}\in S$ and
$\sigma_S(rcr^{-1})=\sigma_S(c)$. The path labeled by $w$
and starting at the vertex $Sr$ is closed, since $Src=Sr$.
Denote its $1$-chain by $\beta_r$.

If $u$ is a word representing $r$, then the path starting at $S$
and labeled by $uwu^{-1}$ has the form
\[
S\xrightarrow{u}Sr\xrightarrow{w}Sr\xrightarrow{u^{-1}}S.
\]
The subpaths corresponding to $u$ and $u^{-1}$ are reverses of one another,
so their chains cancel. Therefore,
\[
\beta_r=\sigma_S(rcr^{-1})=\alpha.
\]
Since $E=RS=SR$, every vertex of $\Gamma_S$ has the form $Sr$.
Thus the path starting at any vertex and labeled by $w$
is closed and has the same nonzero chain $\alpha$.

Fix a positively oriented edge $e$ whose coefficient
in $\alpha$ is nonzero. Each of these paths traverses
$e$ at least once, in one direction or the other.

For the word
\[
w=s_1\cdots s_\ell,\qquad s_j\in X^{\pm1},
\]
set $p_0=1$ and $p_j=s_1\cdots s_j$ for $1\le j\le\ell$.
Suppose that the edge $e$ is labeled by $x\in X$.
If the $j$th step of the path starting at a vertex $v$ traverses $e$
in the positive direction, then $s_j=x$ and
\[
vp_{j-1}=o(e).
\]
If $e$ is traversed in the opposite direction,
then $s_j=x^{-1}$ and
\[
vp_{j-1}=t(e).
\]
Consequently,
\[
V(\Gamma_S)\subseteq
\{\,o(e)p_{j-1}^{-1}:1\le j\le\ell,\ s_j=x\,\}
\cup
\{\,t(e)p_{j-1}^{-1}:1\le j\le\ell,\ s_j=x^{-1}\,\}.
\]
Thus each $j$ corresponds to at most one possible starting
vertex, and hence
\[
[E:S]=|V(\Gamma_S)|\le\ell.\qedhere
\]
\end{proof}

\section{Free factors and first homology}

A subgroup $A\le B$ is called a free factor of $B$ if
\[
B=A*C
\]
for some subgroup $C\le B$.

\begin{KSlemma}[{Karrass--Solitar~\cite[Lemma~1]{KarrassSolitar}}]
Let $G$ be a group, and let $H\le\overline H\le G$ and $K\le G$.
If $H$ is a free factor of $\overline H$, then
$H\cap K$ is a free factor of $\overline H\cap K$.
\end{KSlemma}

Now let $E$ be a free group with a free factor $L$,
and let $H\le E$. Applying the preceding lemma to
$L\le E$ and the subgroup $H$, we see that $H\cap L$
is a free factor of $E\cap H=H$.
For some subgroup $K\le H$, we have
\[
H=(H\cap L)*K.
\]
By the universal property of the free product, there exists
a homomorphism
\[
\rho:H\longrightarrow H\cap L,
\qquad
\rho|_{H\cap L}=\operatorname{id}_{H\cap L},
\qquad
\rho(k)=1\quad(k\in K).
\]

Since $H\cap L\le H$, we have $(H\cap L)'\le H'$.
Together with $(H\cap L)'\le H\cap L\le L$, this gives
\[
(H\cap L)'\le H'\cap L.
\]

Conversely, let $x\in H'\cap L$. Then $x\in H\cap L$,
so $\rho(x)=x$. Since $\rho$ is a homomorphism and
$\rho(H)=H\cap L$, we obtain
\[
x=\rho(x)\in\rho(H')=\rho(H)'=(H\cap L)'.
\]
Thus $H'\cap L\le(H\cap L)'$, and hence
\begin{equation}\label{eq:free-factor}
H'\cap L=(H\cap L)'.
\end{equation}

Moreover, for the inclusion map
$i:H\cap L\hookrightarrow H$, we have
$\rho\circ i=\operatorname{id}_{H\cap L}$. Therefore,
\[
H_1(\rho;\Q)\circ H_1(i;\Q)
=\operatorname{id}_{H_1(H\cap L;\Q)},
\]
and consequently
\[
H_1(i;\Q):H_1(H\cap L;\Q)\longrightarrow H_1(H;\Q)
\]
is a monomorphism.

\begin{lemma}\label{lem:homology}
Let $E$ be a noncyclic free group, and let $H\le E$
be a subgroup containing a nontrivial normal subgroup of $E$.
If the inclusion map $i:H\hookrightarrow E$ induces a monomorphism
\[
H_1(i;\Q):H_1(H;\Q)\longrightarrow H_1(E;\Q),
\]
then $H=E$.
\end{lemma}

\begin{proof}
Let $1\ne D\triangleleft E$, $D\le H$, and $1\ne d\in D$.
For an arbitrary $g\in E$, choose a free factor $L\le E$
of finite rank $n\ge2$ containing $d$ and $g$, and set
$H_0=H\cap L$.

Since $H_0$ is a free factor of $H$, the composite
\[
H_1(H_0;\Q)\longrightarrow H_1(H;\Q)
\longrightarrow H_1(E;\Q)
\]
is injective. It factors through $H_1(L;\Q)$, so
the map
\[
H_1(H_0;\Q)\longrightarrow H_1(L;\Q)
\]
is also injective. Since $L$ is a free group of rank $n$,
we have $H_1(L;\Q)\cong\Q^n$, and since $H_0$ is also free, we obtain
\[
\rank H_0
=\dim_{\Q}H_1(H_0;\Q)
\le \dim_{\Q}H_1(L;\Q)
=n.
\]
In particular, $H_0$ is finitely generated.

Moreover,
\[
1\ne D\cap L\triangleleft L,
\qquad D\cap L\le H_0.
\]
By~\cite[Theorem~2]{KarrassSolitar}, we have $[L:H_0]<\infty$.
Let $m=[L:H_0]$. The Schreier index formula~\cite{MKS} gives
\[
1+m(n-1)=\rank H_0\le n.
\]
Hence $m=1$, so $g\in L=H_0\le H$.
Since $g$ was arbitrary, we conclude that $H=E$.
\end{proof}

\section{The finite-index case}

\begin{lemma}\label{lem:finite-index}
Let $E$ be a noncyclic free group, and suppose that
$R\triangleleft E$, $S\le E$ and $[E:S]<\infty$.
If $R'\le S'$, then $R\le S$.
\end{lemma}

\begin{proof}
Set
\[
N=\core_E(S)=\bigcap_{g\in E}g^{-1}Sg,
\qquad \Omega=E/N,\qquad P=S/N,\qquad A=RN/N.
\]
The group $\Omega$ is finite. Consider
\[
V=(N/N')\otimes_{\Z}\Q
\]
as a right $\Q\Omega$-module with the action induced by conjugation.
For $B\le\Omega$, let $I_B$ denote the augmentation ideal
of $\Q B$, and let $VI_B$ denote the linear span of the elements
$v(b-1)$, where $v\in V$ and $b\in B$.

The five-term exact sequence in homology
\cite[p.~47]{Brown} for the extension corresponding to $N\triangleleft S$ gives
\[
H_2(P;\Q)\longrightarrow V/VI_P
\longrightarrow H_1(S;\Q)
\longrightarrow H_1(P;\Q)\longrightarrow0.
\]
Since $P$ is finite, $H_1(P;\Q)=H_2(P;\Q)=0$.
Hence
\begin{equation}\label{eq:coinvariants}
\ker\bigl(V\longrightarrow H_1(S;\Q)\bigr)=VI_P.
\end{equation}

Suppose now that $n=\rank E<\infty$.
Tensoring the exact sequence for the relation module
\cite[p.~470]{Blackburn} with $\Q$ gives
\begin{equation}\label{eq:relation-sequence}
0\longrightarrow V\longrightarrow(\Q\Omega)^n
\longrightarrow I_\Omega\longrightarrow0.
\end{equation}
The semisimplicity of $\Q\Omega$ and the decomposition
$\Q\Omega\cong\Q\oplus I_\Omega$ yield
\begin{equation}\label{eq:regular-summand}
V\cong\Q\oplus(\Q\Omega)^{n-1}.
\end{equation}
Since $n\ge2$, there exists a surjective homomorphism
of right $\Q\Omega$-modules
\[
\pi:V\longrightarrow\Q\Omega.
\]

The normality of $N$ and $R$ in $E$ implies
\[
[N,R]\le N\cap R,
\qquad [[N,R],R]\le N\cap R'\le N\cap S'.
\]
Let $u\in N$, $r_1,r_2\in R$, $a_i=r_iN\in A$ and
$v=(uN')\otimes1$. The image of $[[u,r_1],r_2]$ in $V$ is
\[
v(a_1-1)(a_2-1).
\]
These elements span $VI_A^2$, and their images
in $H_1(S;\Q)$ vanish. By~\eqref{eq:coinvariants},
we obtain
\begin{equation}\label{eq:ideal-square}
VI_A^2\subseteq VI_P.
\end{equation}
For the finite group $A$, we have $I_A^2=I_A$.
Indeed,
\[
e_A=\frac1{|A|}\sum_{a\in A}a,
\qquad I_A=\Q A(1-e_A),
\qquad (1-e_A)^2=1-e_A.
\]
Thus~\eqref{eq:ideal-square} gives $VI_A\subseteq VI_P$.
Since $\pi(VI_B)=\Q\Omega I_B$ for every $B\le\Omega$,
we obtain
\begin{equation}\label{eq:coset-ideals}
\Q\Omega I_A\subseteq\Q\Omega I_P.
\end{equation}

We have a natural isomorphism of $\Q$-vector spaces
\[
\Q\Omega/\Q\Omega I_P\cong \Q[\Omega/P].
\]
For $a\in A$, it follows from~\eqref{eq:coset-ideals} that
$a-1\in\Q\Omega I_P$, so $aP=P$.
Therefore $A\le P$, and hence $R\le S$.

It remains to consider the case $\rank E=\infty$.
For each $r\in R$, choose a free factor $L\le E$
of finite rank at least two containing $r$. Then
\[
R\cap L\triangleleft L,
\qquad [L:S\cap L]\le[E:S]<\infty.
\]
By~\eqref{eq:free-factor},
\[
(R\cap L)'=R'\cap L
\le S'\cap L=(S\cap L)'.
\]
Applying the case already proved to $R\cap L$ and $S\cap L$, we obtain
$r\in S\cap L$. Thus $R\le S$.
\end{proof}

\section{Proof of the main theorem}

\begin{proof}[Proof of Theorem~\ref{thm:main}]
The case $R=1$ is immediate; henceforth assume $R\ne1$.
Then $R$ is noncyclic, so $R'\ne1$. Set
\[
E=RS,\qquad C=R\cap S.
\]
Then $E$ is a noncyclic free group and $R\triangleleft E$.

If $C\le S'$, then $(S/C)'=S'/C$. Moreover,
\[
S/C\cong E/R.
\]
Hence
\[
H_1(S;\Q)\cong H_1(S/C;\Q)\cong H_1(E/R;\Q).
\]
This isomorphism is induced by the composite
$S\hookrightarrow E\twoheadrightarrow E/R$. Consequently, the induced homomorphism
\[
H_1(S;\Q)\longrightarrow H_1(E;\Q)
\]
is a monomorphism.
Furthermore, $1\ne R'\triangleleft E$ and $R'\le S$.
Lemma~\ref{lem:homology}, applied with $H=S$ and $D=R'$,
gives $S=E$, whence $R\le S$.

If $C\nleq S'$, then Lemma~\ref{lem:support}
shows that the index $[E:S]$ is finite. Applying Lemma~\ref{lem:finite-index},
we again obtain $R\le S$.
\end{proof}

\begin{remark}
The normality hypothesis on $R$ in Theorem~\ref{thm:main} is essential.
In Dunwoody's example~\cite[p.~153]{Dun}, one takes
\[
F=\langle a,b\rangle,\qquad
U=\langle ba,b^2\rangle,\qquad
W=\langle a,b^2,bab^{-1}\rangle.
\]
The subgroup $W$ is the kernel of the homomorphism
$F\to\Z/2\Z$ defined by $a\mapsto0$, $b\mapsto1$.
Since $W'\triangleleft F$ and
\[
[ba,b^2]=[a,b^2]\in W',
\]
the image of $U$ in $F/W'$ is abelian, so $U'\le W'$.
However, $ba\notin W$, and hence $U\nleq W$.

The hypothesis that $F$ is noncyclic is also essential: for
$F=R=\langle a\rangle$ and $S=\langle a^2\rangle$,
we have $R'=S'=1$, but $R\nleq S$.
\end{remark}

\end{document}